\documentclass[12pt]{article}
\usepackage[utf8]{inputenc}

\usepackage{amsmath}
\usepackage{amssymb}
\usepackage{amsthm}
\usepackage{mathrsfs}
\usepackage{mathtools}
\usepackage{tikz}
\usepackage{cleveref}
\usepackage{tikz-cd}
\usepackage{mathtools}
\usepackage{verbatim}
\usepackage{appendix}
\usepackage{xcolor}
\usepackage[margin = 1 in]{geometry}

\usepackage{setspace}
\newcommand{\rood}[1]{\textcolor{red}{[#1]}}

\theoremstyle{definition}
\newtheorem{thm}{Theorem}[section]
\crefname{thm}{Theorem}{Theorems}

\newtheorem{prop}[thm]{Proposition}
\crefname{prop}{Proposition}{Propositions}
\newtheorem{lem}[thm]{Lemma}
\crefname{lem}{Lemma}{Lemmas}

\newtheorem{conj}[thm]{Conjecture}

\crefname{defn}{Definition}{Definitions}

\newtheorem{exmp}[thm]{Example}

\newtheorem{rmk}[thm]{Remark}

\newtheorem*{ack*}{Acknowledgements}

\newcommand{\on}{\operatorname}
\newcommand{\mb}{\mathbb}
\newcommand{\ol}{\overline}
\newcommand{\mc}{\mathcal}

\newcommand{\mf}{\mathfrak}

\newcommand{\wt}{\widetilde}

\newenvironment{polynomial}
  {\par\vspace{\abovedisplayskip}%
   \setlength{\leftskip}{\parindent}%
   \setlength{\rightskip}{\leftskip}%
   \medmuskip=4mu plus 2mu minus 2mu
   \binoppenalty=0
   \noindent$\displaystyle}
  {$\par\vspace{\belowdisplayskip}}

\title{Divisors on the moduli space of curves from divisorial conditions on hypersurfaces}
\author{Dennis Tseng}
\date{\today}

\begin{document}

\maketitle

\begin{abstract}
In this note, we extend work of Farkas and Rim\'anyi on applying quadric rank loci to finding divisors of small slope on the moduli space of curves by instead considering all divisorial conditions on the hypersurfaces of a fixed degree containing a projective curve. This gives rise to a large family of virtual divisors on $\ol{\mc{M}_g}$. We determine explicitly which of these divisors are candidate counterexamples to the Slope Conjecture. The potential counterexamples exist on $\ol{\mc{M}_g}$, where the set of possible values of $g\in \{1,\ldots,N\}$ has density $\Omega(\log(N)^{-0.087})$ for $N>>0$. Furthermore, no divisorial condition defined using hypersurfaces of degree greater than 2 give counterexamples to the Slope Conjecture, and every divisor in our family has slope at least $6+\frac{8}{g+1}$. 
\end{abstract}

\section{Introduction}
Given an effective divisor $D$ on the moduli space of curves $\ol{\mc{M}_g}$, there is an associated number $s(D)$ called the \emph{slope} of the divisor. There has been significant interest in finding effective divisors on $\ol{\mc{M}_g}$ of small slope in order to give an upper bound for the slope
\begin{align*}
   s(\ol{\mc{M}_g})=\inf\{s(D): D\text{ is effective}\}
\end{align*}
of $\ol{\mc{M}_g}$. The slope $s(\ol{\mc{M}_g})$ gives information about the effective cone of $\ol{\mc{M}_g}$ and upper bounds on $s(\ol{\mc{M}_g})$ have been applied to show $\ol{\mc{M}_g}$ is general type for sufficiently large $g$ \cite{HM82,H84,EH87,F09,FOTN18,JP18,FJP20}. 

To tell whether our computed slope $s(D)$ is small, the standard is to compare it with the Slope Conjecture \cite[Conjecture 0.1]{HM90} stated by Harris and Morrison
\begin{conj}
The slope $s(\ol{\mc{M}_g})$ is always at least $6+\frac{12}{g+1}$. 
\end{conj}
If $g+1$ is composite, then the Brill-Noether divisors achieve a slope of $6+\frac{12}{g+1}$. Since then, there have been counterexamples discovered for the Slope Conjecture for some values of $g$ \cite{FP05,F06,Koszul,Quadric,Khosla}. Given the difficulty in constructing the examples, there is value in understanding what sorts of divisors should give rise to small slopes. 

In this note, we focus on extending the methods of Farkas and Rim\'anyi \cite{Quadric}, which affords a large class of divisors on $\ol{\mc{M}}_g$, given by divisorial conditions on the hypersurfaces containing a projective embedding of genus $g$ curves. 
In \cite{Quadric}, the authors fixed $g,r,d$ so the Brill Noether number $\rho=g-(r+1)(r-d+g)$ is $0$ and asked for nondegenerate curves $C\to \mb{P}^r$ of degree $d$ and genus $g$ to either lie on a quadric of low rank or be contained in a degenerate pencil of quadrics. When either of these two conditions is a divisorial condition on the space of quadrics containing $C$, one gets a (virtual) divisor on $\ol{\mc{M}}_g$. The authors exhibited infinitely many examples of potential counterexamples to the Slope Conjecture and verified the potential counterexamples were actual divisors in small cases using \emph{Macaulay} \cite[Section 7]{Quadric}.

Our contribution is twofold. First, we show their argument can both be easily simplified and generalized to apply to \emph{any} divisorial condition on the hypersurfaces of degree $m\geq 2$ containing a curve (see \Cref{divisorhypersurface}). Second, we use the formulas to deduce three results (see \Cref{main}):
\begin{enumerate}
    \item We show the slopes of all our divisors are all bounded below by $6+\frac{8}{g+1}$. This gives evidence that $s(\ol{\mc{M}}_g)$ approaches $6$ as $g\to \infty$ in the context of \cite[Problem 0.1]{CFM13}.
    \item Only divisors defined using quadrics (instead of hypersurfaces of higher degree) can give counterexamples to the Slope Conjecture.
    \item We give virtual divisors that are candidate counterexamples to the Slope Conjecture on $\ol{\mc{M}_g}$ for all $g=(r+1)s$ with $\frac{r^2+1}{3r-1}< s\leq \frac{r}{2}$.
\end{enumerate}
As intuition, from our calculations it appears that simpler is better for small slope. This is supported by the second point above regarding divisors defined using hypersurfaces of higher degree. Furthermore, it also appears that among conditions on quadrics, the condition of simply being contained in a single quadric gives the best slope (see \Cref{quadricwin}). 

Perhaps another source of intuition can be gained by looking at $K3$ surfaces. Namely, it is known that any divisor violating the Slope Conjecture must contain all genus $g$ curves occurring as a hyperplane section of a $K3$ surface in $\mb{P}^g$ of degree $2g-2$ \cite[Proposition 2.2]{FP05}. While \cite[Theorem 1.7]{FP05} gives such a condition involving quadrics when $g=10$, it seems more difficult to produce conditions of higher degree that are always satisfied by such curves. 

\subsection{Statement of results}
\subsubsection{Definition of slope}
We recall for $g\geq 3$ \cite[Theorem 1]{AC87},
\begin{align*}
    A^{\bullet}(\ol{\mc{M}_{g}})\otimes \mb{Q}=\mb{Q}\lambda\oplus \bigoplus_{i=0}^{\lfloor \frac{g}{2}\rfloor}{\mb{Q}\delta_i}.
\end{align*}
Given an effective divisor $D=a\lambda -\sum_{i=0}^{\lfloor \frac{g}{2}\rfloor}{b_i\delta_i}$ on $\ol{\mc{M}_g}$ with $a,b_i> 0$, define the slope 
\begin{align*}
    s(D) &= \frac{a}{\min\{b_i: 0\leq i\leq \lfloor \frac{g}{2}\rfloor\}}.
\end{align*}
If $a,b_i$ are not all positive, then we define $s(D)=\infty$. Define $s(\ol{\mc{M}_g})$ to be the infimum of $s(D)$ as $D$ varies over all effective divisors. 

Even though this is not standard, we will similarly define
\begin{align*}
    s_0(D) &= \begin{cases}
    \frac{a}{b_0} \qquad\text{if $a,b_0>0$}\\
    \infty \qquad\text{otherwise}\\
    \end{cases}\\
    s_0(\ol{\mc{M}_g})&=\inf\{s_0(D): D\text{ is effective}\}.
\end{align*}

Clearly, $s_0(D)\leq s(D)$ and $s_0(\ol{\mc{M}_g})\leq s(\ol{\mc{M}_g})$. Conjecturally, $s_0(\ol{\mc{M}_g})= s(\ol{\mc{M}_g})$ \cite[Conjecture 1.5]{FP05}, and this has been verified for $g\leq 23$ \cite[Theorem 1.4]{FP05}. For technical reasons (see \Cref{setup}), we will work with $s_0(D)$ instead of $s(D)$, which means our candidate counterexamples to the Slope Conjecture have only been checked with the coefficients of $\lambda$ and $\delta_0$. The slopes computed in \cite{Quadric} also are only computed with $s_0(D)$ instead of $s(D)$, and it would be interesting to find a way to compute the other coefficients. 
\subsubsection{Definition of the divisors}
\label{Definitions}
We will work with $\ol{\mc{M}_g}$ as a Deligne-Mumford stack instead of a coarse moduli space, but the distinction does not matter for the statement of \Cref{main}. We will work over $\mb{C}$, but see \Cref{char} for more on characteristic assumptions on the base field. Fix $r,g,d$ such that the Brill-Noether number $\rho:= g-(r+1)(g-d+r)$ is zero. Equivalently, we have $s\geq 1,r\geq 1$ such that $g=(r+1)s$, $d=(s+1)r$. Since we are interested in the hypersurfaces containing a curve $C\to\mb{P}^r$, we also assume $r\geq 3$. Given an integer $m\geq 2$ such that $\binom{r+m}{m}\geq md-g+1$, fix a divisor $D\subset \on{Hom}(\on{Sym}^m\mb{C}^{r+1},\mb{C}^{md-g+1})$ invariant under the action of $GL(\mb{C}^{r+1})\times GL(\mb{C}^{md-g+1})$. We give examples of such invariant divisors in \Cref{sec:example} and give more about the general classification in \Cref{classdivisors}.

Let $\mc{M}_g^{\on{irr}}\subset \ol{\mc{M}_g}$ denote the open substack parameterizing irreducible curves of genus $g$. In $\mc{M}_g^{\on{irr}}$, consider the locus $Z^{m,r,s}_g$ consisting of curves $C$ for which there exists a line bundle $L$ of degree $d$ mapping $C\to \mb{P}^r$ such that the induced map
\begin{align*}
    H^0(\mb{P}^r,\mathscr{O}_{\mb{P}^r}(m))\to H^0(C,L^{\otimes m})
\end{align*}
is given by a map in $D$ after choosing bases for $H^0(L)$ and $H^0(C,L^{\otimes m})$. Since $D$ is invariant under $GL(\mb{C}^{r+1})\times GL(\mb{C}^{md-g+1})$, this definition is independent of choice of bases. 

Now, take the closure of $Z_{m,r,s}$ in $\ol{\mc{M}}_g$ to get $D_{m,r,s}\subset \ol{\mc{M}}_g$. If $Z_{m,r,s}$ is not dense, then $D_{m,r,s}$ is a divisor and we can compute its slope. Otherwise, we only know its slope as a virtual divisor. In either case, we can define $s(D_{m,r,s})$ and $s_0(D_{m,r,s})$ as above. 

\subsubsection{Main results}
Now, we state our main theorem

\begin{thm}
\label{main}
The slope $s_0(D_{m,r,s})$ is independent of the choice of the $GL(\mb{C}^{r+1})\times GL(\mb{C}^{md-g+1})$-invariant divisor $D\subset \on{Hom}(\on{Sym}^m\mb{C}^{r+1},\mb{C}^{md-g+1})$ given $m\geq 2,r\geq 3,s\geq 1$. Furthermore
\begin{enumerate}
    \item If $m\geq 3$, $s_0(D_{m,r,s})\geq 6+\frac{12}{g+1}$, so considering hypersurfaces other than quadrics will not yield counterexamples to the Slope Conjecture. Equality holds if and only if $(m,r,s)=(3,3,2)$. 
    \item We have $s_0(D_{2,r,s})> 6+\frac{8}{g+1}$.
    \item We have $s_0(D_{2,r,s})< 6+\frac{12}{g+1}$ if and only if $\frac{r^2+1}{3 r-1}<s\leq \frac{r}{2}$. 
\end{enumerate}
\end{thm}
\noindent 
For $N>>0$, the values of $g$ in $\{1,\ldots,N\}$ for which \Cref{main} produces a potential counterexample has density $\Theta(\frac{1}{\log(g)^{\delta}\log(\log(x))^{\frac{3}{2}}})$ by \cite[Corollary 2]{F08}, where $\delta=1-\frac{1+\log(\log(2))}{\log(2)}\approx .086071$.

The main ingredient in the proof of \Cref{main} is \Cref{dualloci} given below. Our proof also involves straightforward, but tedious, formula manipulation using Mathematica, given in \Cref{grind}.

\begin{thm}
\label{dualloci}
Let $D\subset \on{Hom}(\on{Sym}^m\mb{C}^e,\mb{C}^f)$ be a divisor, preserved under the natural actions of $GL(\mb{C}^e)$ and $GL(\mb{C}^f)$. Given vector bundles $\mc{E}$ and $\mc{F}$ of ranks $e$ and $f$ respectively over a scheme $B$ together with a map $\phi: \on{Sym}^m\mc{E}\to \mc{F}$, the class of the virtual divisor supported on points of $B$ over which $\phi$ fiberwise restricts to maps in $D$ is a positive rational multiple of
\begin{align*}
    ec_1(\mc{F})-mfc_1(\mc{E}).
\end{align*}
\end{thm}

\Cref{dualloci} generalizes \cite[Theorems 1.1 and 1.2]{Quadric} in the setting of divisors, and follows from standard methods of equivariant intersection theory. We note that \cite[Theorems 1.1 and 1.2]{Quadric} are more precise in that their formulas are exact, instead of up to a scalar multiple like \Cref{dualloci}, and that \cite[Theorems 1.1]{Quadric} applies to quadric rank loci, which includes settings where the codimension is greater than 1.

We will be applying \Cref{dualloci} in the case where $B$ is the Deligne-Mumford stack of the moduli space of curves. To do so, one can pull back to enough test curves, for example the test curves given in \cite[Table 3.141]{HM98}. Alternatively, one can pull back to a finite cover of $\ol{\mc{M}}_g$, given in \cite[Lemma 3.89]{HM98} or \cite[Proposition 2.6]{V89}.

\subsection{Example cases and comparison to literature}
\label{sec:example}
\begin{exmp}
If $\binom{r+m}{m}=md-g+1$, then the unique choice of invariant divisor $D\subset \on{Hom}(\on{Sym}^m\mb{C}^{r+1},\mb{C}^{md-g+1})$ consists of linear maps that are not of full rank. This is the locus of curves contained in a degree $m$ hypersurface. In the case $(r,g,d,m)=(4,10,12,2)$ this is the first known counterexample to the Slope Conjecture \cite[Theorem 1.7(4)]{FP05} and this was considered in general by Khosla \cite[Section 3-B]{Khosla}. For the case of $m=2$, it has been checked that the coefficient of each $\delta_i$ for $i>0$ does not contribute to the slope \cite[Theorem 1.4]{Koszul}.
\end{exmp}

\begin{exmp}
The case $(r,g,d,m)=(5,12,15,2)$ was considered in \cite[Section 8]{Quadric}. Given a general genus 12 curve $C$ together with one of its finitely many degree 15 embeddings $C\subset \mb{P}^5$, there is a pencil of quadrics containing it. Pulling back the discriminant hypersurface of singular quadrics yields 6 points (possible non-distinct) on $\mb{P}^1$. To illustrate the independence of the slope on the choice of divisor $D\subset \on{Hom}(\on{Sym}^2\mb{C}^{6},\mb{C}^{19})$ in the statement of \Cref{main}, the following divisorial conditions on those 6 points yield virtual divisors on $\ol{\mc{M}}_{12}$ with the same slopes, each contradicting the Slope Conjecture. To bound the coefficients of $\delta_i$ for $i>0$, one can use \cite[Corollary 1.2]{FP05}.
\begin{enumerate}
    \item 6 points on $\mb{P}^1$, where at least two points coincide. This was considered in \cite[Section 8]{Quadric} and shown to be an actual divisor using \emph{Macaulay}.
    \item
    6 points on $\mb{P}^1$ with an involution.
    \item
    6 points on $\mb{P}^1$ such that 4 of them have a fixed choice of moduli.
    \item
    6 points on $\mb{P}^1$ that arise as the image of 6 points on $\mb{P}^2$ under a linear map $\mb{P}^2\dashrightarrow \mb{P}^1$. It is not necessary for the 6 points to be general. For example, it suffices for 5 of them to be in general linear position. 
\end{enumerate}
\end{exmp}

\begin{exmp} Let $m=2$. 
If $r=9\ell-2$ and $s=4\ell-1$, this recovers \cite[Theorem 7.1]{Quadric}, and similarly if $r=8\ell+3$ and $s=3\ell+1$, this recovers \cite[Theorem 7.2]{Quadric}. The authors state the result in terms of $s(D_{2,9\ell-2,4\ell-1})$ and $s(D_{2,8\ell+3,3\ell+1})$, but they also only computed $s_0(D_{2,9\ell-2,4\ell-1})$ and $s_0(D_{2,8\ell+3,3\ell+1})$.
\end{exmp}

\begin{exmp}
The smallest case of \Cref{main} that is new to our knowledge is when $(g,r,d)$ is $(27,8,32)$. Given a line bundle $L$ of degree 32 mapping a genus 27 curve $C\to \mb{P}^8$, we expect $\dim(\on{Sym}^2H^0(C,L))=45$ and $H^0(C,L^{\otimes 2})=38$, so we expect a $\mb{P}^6$ of quadrics containing $C$, and there to be $\frac{\binom{9}{3}\binom{10}{2}\binom{11}{1}}{\binom{1}{0}\binom{3}{1}\binom{5}{2}}=1386$ quadrics of corank at least 3 \cite[Proposition 12(b)]{HT84}. We can choose $D$ to be the divisor where at least two of these points coincide.
\end{exmp}

\begin{exmp}
\label{quadricwin}
If $g=10+6i$ for $i\geq 0$, then \cite[Theorem A]{F06} gives a virtual counterexample to the Slope Conjecture, where the coefficients of $\delta_i$ for $i>0$ are also checked. There are cases where \Cref{main} and \cite[Theorem A]{F06} overlap, and computing the slopes with Mathematica in small cases suggests that the divisor computed in \cite[Theorem A]{F06} will always have smaller slope unless $r=12\ell$, $s=6\ell$, and $\binom{r+2}{2}=2d-g+1$. In this case, $D$ corresponds to curves lying on a quadric hypersurface.
This has been tested for all values of $r<1000$.
\end{exmp}

\begin{exmp}
In the equality case of Part 1 of \Cref{main}, we are looking at genus 8 curves with a degree 9 map $C\to \mb{P}^3$ contained in a cubic surface. This set-theortically contains the Brill-Noether divisor of curves with a $\mathfrak{g}^2_7$. Suppose we have $f: C\to \mb{P}^2$ whose image is a septic plane curve with 7 nodes. The canonical divisor on the image is $4L$, where $L$ is the class of a line in $\mb{P}^2$. The canonical divisor of $C$ is then $4f^{*}L-\sum_{i=1}^{7}{(p_i+q_i)}$, where $p_i,q_i$ are the preimages of the 7 nodes. Pick one of the nodes, for example the node corresponding to $p_7,q_7$. The lines through that node give a $\mf{g}_{5}^1$ on $C$. Subtracting this $\mf{g}_{5}^1$ from the canonical on $C$ gives $3f^{*}L-\sum_{i=1}^{6}{(p_i+q_i)}$, which is the cubics in $\mb{P}^2$ passing through the other 6 nodes of the image of $C$. This gives $C\to \mb{P}^2\dashrightarrow\mb{P}^3$ which yields a degree 9 embedding of $C$ into $\mb{P}^3$ contained in a cubic surface. The class of $C$ on the cubic surface is $7L-2(E_1+\cdots+E_6)$. It is not immediately clear to us, for example, whether curves corresponding to $9L-3(E_1+\cdots+E_6)$ or $11L-4(E_1+\cdots+E_6)$ could also contribute additional components to $D_{3,3,2}$.
\end{exmp}

In light of these examples, it is natural to ask for a classification of the divisors $D\subset \on{Hom}(\on{Sym}^m(V),W)$, which are invariant under the action of $GL(V)\times GL(W)$. In particular, if one wants to show that we can choose $D$ so that $D_{m,r,s}$ in \Cref{main} is not a virtual divisor, then it makes sense to consider the intersection of all such invariant divisors $D$. In \Cref{classdivisors}, we observe that the intersection of all invariant divisors $D$ coincides with a locus of GIT unstable points. 

\subsection{A note on characteristic assumptions}
\label{char}
We will work over $\mb{C}$ for notational convenience, but our proofs are algebraic, so everything automatically extends to when our base field is an algebraically closed field of characteristic zero. 

\Cref{divisorhypersurface} holds independent of characteristic. To extend \Cref{main} to positive characteristic, one would need to check that the setup in \cite{Khosla} or \cite[Section 2]{Koszul} to pushforward classes from the stack parameterizing curves with a linear series to the moduli space of curves can be adapted to positive characteristic. The Picard group $\on{Pic}(\ol{\mc{M}_{g,n}})\otimes\mb{Q}$ is unchanged in positive characteristic \cite{M01}. More seriously, when applying limit linear series arguments in positive characteristic, we want to restrict ourselves to cases where ramification is imposed at at most two points on each component \cite{O14,O18}. For example, since we only compute the coefficients of $\lambda$ and $\delta_0$, \cite[Lemma 4.5]{Khosla} suffices for our use, but Khosla degenerates further to a comb of elliptic curves with a rational backbone in the proof.

\section{Divisors from hypersurfaces}
\label{divisorhypersurface}
The goal of this section is to prove \Cref{dualloci}. For completeness, we also state the dual version of \Cref{dualloci} in \Cref{loci} below.

\begin{thm}
\label{loci}
Let $D\subset \on{Hom}(\mb{C}^e,\on{Sym}^m\mb{C}^f)$ be a divisor, preserved under the natural actions of $GL(\mb{C}^e)$ and $GL(\mb{C}^f)$. Given vector bundles $\mc{E}$ and $\mc{F}$ of ranks $e$ and $f$ respectively over a scheme $B$ together with a map $\phi: \mc{E}\to \on{Sym}^m\mc{F}$, the class of the virtual divisor supported on points of $B$ over which $\phi$ fiberwise restricts to maps in $D$ is a positive rational multiple of
\begin{align*}
    mec_1(\mc{F})-fc_1(\mc{E}).
\end{align*}
\end{thm}

\Cref{loci} is perhaps more natural to state, but we will not need to apply it in this note. Therefore, we will only prove \Cref{dualloci} while noting the proof of \Cref{loci} is completely analogous. 

\subsection{Equivariant intersection theory}
The first step of the proof of \Cref{dualloci} is to identify the answer in terms of the corresponding equivariant class, which is given in \Cref{lem:equivariantreduction} below. While the argument might be standard to experts, we attempt to review the relevant background and give more details. In particular, all our equivariant arguments can in principle be reduced to usual intersection theory.

\subsubsection{Basic properties and construction of equivariant intersection theory}
Given a smooth scheme $X$ with an action of an algebraic group $G$, the equivariant Chow ring $A^{\bullet}_G(X)$ is the integral Chow ring of the quotient stack $[X/G]$ in the sense of operational Chow rings \cite[Proposition 19]{EG98}. More precisely, a class $\alpha \in A^{k}([X/G])$ defines a map $A_{\bullet}(B)\xrightarrow{\alpha_f} A_{\bullet-k}(B)$ for each map $f: B\to [X/G]$ from a scheme into $[X/G]$. As in \cite[Definition 17.1]{F98} and the text right before \cite[Proposition 19]{EG98}, these maps $\alpha_f$ need to satisfy certain compatibility properties with respect to proper pushforward, flat pullback and intersection products. For the purposes of brevity, we will work with this definition of the equivariant Chow ring when possible. However, we will occasionally have to refer to the construction of the equivariant Chow groups in \cite{EG98}, which gives a more concrete description more amenable to computation. The construction roughly goes as follows.

Given a principal $G$-bundle $\mathcal{P}\to B$, we can form the quotient $X\times^{G}\mathcal{P}$, which is an $X$-bundle over $B$. The construction in \cite{EG98} defines the equivariant Chow group $A_{k}^G(X)$ as $A_{k+\dim(B)-\dim(G)}(X\times^{G}\mathcal{P})$, where we choose $\mathcal{P}\to B$ to be a quotient $U\to U/G$ where $U\subset \mathbb{A}^N$ is an open subset in a representation of $G$ whose complement has large codimension and $G$ acts freely on $U$. In our specific situation, $X$ is an affine space and $G$ is a product of general linear groups, so $A_{k}^G(X)\cong A_k^G(\operatorname{pt})$ turns out to be the Chow group of a vector bundle over a product of Grassmannians. The operational Chow ring $A^{\bullet}_G(X)$ defined in \cite[Section 2.6]{EG98}, is isomorphic to the $A_{\bullet}^{G}(X)$ by Poincar\'e duality since $X$ is smooth \cite[Proposition 4]{EG98}. 

A good exposition of this point of view of equivariant intersection theory in terms of approximating $BG$ is given in \cite{A12}. In particular, we can always reduce statements about equivariant intersection theory to usual intersection theory.\footnote{There are some mild conditions required for $X\times^G\mathcal{P}$ to be a scheme, given in \cite[Proposition 23]{EG98}. Since products of general linear groups are special, they are always satisfied in our case.}

\subsubsection{Relevant equivariant Chow rings}
We are interested in the following equivariant Chow rings.

\begin{prop}
\label{prop:presentation}
The ring $A^{\bullet}_{GL(\mb{C}^e)}(\on{pt})$ is isomorphic to $\mb{Z}[c_1,\ldots,c_e]$, generated by the chern classes of the universal vector bundle $[\mb{C}^e/GL(\mb{C}^e)]\to [\on{pt}/GL(\mb{C}^e)]$. Furthermore, pulling back the chern classes under the maps of $[\on{Hom}(\on{Sym}^m\mb{C}^e,\mb{C}^f)/GL(\mb{C}^e)\times GL(\mb{C}^f)]$ to $[\on{pt}/GL(\mb{C}^e)]$ and $[\on{pt}/GL(\mb{C}^e)]$ yields an isomorphism
\begin{align*}
   A^{\bullet}_{GL(\mb{C}^e)\times GL(\mb{C}^f)}(\on{Hom}(\on{Sym}^m\mb{C}^e,\mb{C}^f))&\cong    A^{\bullet}_{GL(\mb{C}^e)}(\on{pt})\otimes_{\mb{Z}}A^{\bullet}_{GL(\mb{C}^f)}(\on{pt})\\
   &\cong  \mb{Z}[c_1^e,\ldots,c_e^e]\otimes_{\mathbb{Z}}\mathbb{Z}[ c_1^f,\ldots,c_f^f].
\end{align*}
\end{prop}

\begin{proof}
For the first statement, $A^{\bullet}([\on{pt}/GL(\mb{C}^e)])$ is a polynomial ring over $\mathbb{Z}$, generated by the chern classes of the universal bundle $[\mb{C}^e/GL(\mb{C}^e)]\to [\on{pt}/GL(\mb{C}^e)]$ \cite[Section 15]{T99}. Put another way, $A^{\bullet}([\on{pt}/GL(\mb{C}^e)])$ is approximated by the Chow ring of the Grassmannian $A^{\bullet}(\on{Gr}(e,N))$ for $N>>0$ from \cite[Section 3.1]{EG98} and \cite[Theorem 2]{EG98}.

The Chow ring of the product 
\begin{align*}
 A^{\bullet}([\on{pt}/GL(\mb{C}^e)]\times [\on{pt}/GL(\mb{C}^f)])&=    A^{\bullet}([\on{pt}/GL(\mb{C}^e)])\otimes_{\mathbb{Z}}A^{\bullet}([\on{pt}/GL(\mb{C}^f)])\\
 &= \mb{Z}[c_1^e,\ldots,c_e^e]\otimes_{\mathbb{Z}}\mathbb{Z}[ c_1^f,\ldots,c_f^f]
\end{align*}
satisfies the K\"unneth formula. To see this, we use the fact stated above that $A^{\bullet}([\on{pt}/GL(\mb{C}^e)])$ and $A^{\bullet}([\on{pt}/GL(\mb{C}^f)])$ are approximated by the Chow rings of the Grassmannians $A^{\bullet}(\on{Gr}(e,N))$ and $A^{\bullet}(\on{Gr}(f,N))$ for $N>>0$. Since Grassmannians are unions of affine spaces, the K\"unneth formula holds in Chow for a product of Grassmannians \cite[Proposition 1]{T14}. 

Since $[\on{Hom}(\on{Sym}^m\mb{C}^e,\mb{C}^f)/GL(\mb{C}^e)\times GL(\mb{C}^f)]$ is a vector bundle over $[\operatorname{pt}/GL(\mb{C}^e)\times GL(\mb{C}^f)]$, we obtain the presentation
\begin{align*}
    A^{\bullet}([\on{Hom}(\on{Sym}^m\mb{C}^e,\mb{C}^f)/GL(\mb{C}^e)\times GL(\mb{C}^f)])& \cong A^{\bullet}([\on{pt}/GL(\mb{C}^e)\times GL(\mb{C}^f)])\\
    & \cong A^{\bullet}([\on{pt}/GL(\mb{C}^e)]\times [\on{pt}/GL(\mb{C}^f)])\\
    &\cong \mb{Z}[c_1^e,\ldots,c_e^e]\otimes_{\mathbb{Z}}\mathbb{Z}[ c_1^f,\ldots,c_f^f].
\end{align*}
In the second isomorphism, we used the isomorphism $[\on{pt}/GL(\mb{C}^e)\times GL(\mb{C}^f)]\cong [\on{pt}/GL(\mb{C}^e)]\times [\on{pt}/GL(\mb{C}^f)]$ as stacks.
\end{proof}

Now that we have the presentation of $   A^{\bullet}_{GL(\mb{C}^e)\times GL(\mb{C}^f)}(\on{Hom}(\on{Sym}^m\mb{C}^e,\mb{C}^f))$ given in \Cref{prop:presentation}, we can now precisely state how to translate from \Cref{dualloci} to an equivariant Chow class. 

\begin{lem}
\label{lem:equivariantreduction}
Let $D\subset \on{Hom}(\on{Sym}^m\mb{C}^e,\mb{C}^f)$ be a divisor, preserved under the natural actions of $GL(\mb{C}^e)$ and $GL(\mb{C}^f)$. Given vector bundles $\mc{E}$ and $\mc{F}$ of ranks $e$ and $f$ respectively over a scheme $B$, let $\mc{D}$ be the divisor $\mc{D}\subset \on{Hom}(\on{Sym}^m\mc{E},\mc{F})$ that restricts fiberwise to $D$.

If $[D]\in A^{\bullet}_{GL_e\times GL_f}(\on{Hom}(\on{Sym}^m\mb{C}^e,\mb{C}^f))$ is given by $a c_1^e+bc_1^f$, then the class of $\mc{D}$ is given by $ac_1(E)+bc_1(F)$.
\end{lem}

\begin{proof}[Proof of \Cref{lem:equivariantreduction}]
The vector bundles $\mc{E}$ and $\mc{F}$ define a map 
\begin{align*}
    B\xrightarrow{f} [\on{pt}/GL(\mb{C}^e)]\times [\on{pt}/GL(\mb{C}^f)]\cong [\on{pt}/GL(\mb{C}^e)\times GL(\mb{C}^f)].
\end{align*}
Under the map $f$, the vector bundle $[\mb{C}^e/GL(\mb{C}^e)\times GL(\mb{C}^f)]\to \left[\on{pt}/GL(\mb{C}^e)\times GL(\mb{C}^f)\right]$ pulls back to $\mc{E}$ and the vector bundle $[\mb{C}^f/GL(\mb{C}^e)\times GL(\mb{C}^f)]\to \left[\on{pt}/GL(\mb{C}^e)\times GL(\mb{C}^f)\right]$ pulls back to $\mc{F}$. 

Pulling back under $f$, we also get the following commutative diagram, where both squares are fiber products.
\begin{center}
\begin{tikzcd}
\mathcal{D} \ar[r] \ar[hook,d] & \left[D/GL(\mb{C}^e)\times GL(\mb{C}^f)\right] \ar[hook, d] \\
\on{Hom}(\on{Sym}^m\mc{E},\mc{F}) \ar[r] \ar[d] & \left[\on{Hom}(\on{Sym}^m\mb{C}^e,\mb{C}^f)/GL(\mb{C}^e)\times GL(\mb{C}^f)\right] \ar[d] \\
B \ar[r,"f"] & \left[\on{pt}/GL(\mb{C}^e)\times GL(\mb{C}^f)\right]
\end{tikzcd}
\end{center}
Thus, the class of $\mc{D}$ is the pullback of the class of the substack $ \left[D/GL(\mb{C}^e)\times GL(\mb{C}^f)\right]$. Since the class of the substack is $a c_1^e+bc_1^f$ pulled back from $\left[\on{pt}/GL(\mb{C}^e)\times GL(\mb{C}^f)\right]$, functorality of chern classes under pullback yields the class of $\mc{D}$ is $ac_1(E)+bc_1(F)$.
\end{proof}

\begin{rmk}
We chose to phrase the proof of \Cref{lem:equivariantreduction} in terms of stacks so the organization is clearer at a high level. However, one could also phrase everything in terms of schemes using the construction of equivariant Chow rings. 

For the benefit of the readers who are interested in how the proof operates at an elementary level, we briefly describe it here. By construction, the class $[D]\in A_{GL_e\times GL_f}^{\bullet}(\on{Hom}(\on{Sym}^m\mb{C}^e,\mb{C}^f))$ is given by the class of $\mc{D}$ in the case $\mc{E}$ and $\mc{F}$ are the pullbacks of the tautological subbundles of $\on{Gr}(e,N)$ and $\on{Gr}(f,N)$ to $\on{Gr}(e,N)\times \on{Gr}(f,N)$ for $N>>0$. 

More generally, if $\mc{E}$ and $\mc{F}$ are subbundles of a trivial bundle, we have an induced map $B\to \on{Gr}(e,N)\times \on{Gr}(f,N)$. Pulling back reduces to the case above. However, if $\mc{E}$ or $\mc{F}$ are not subbundles of a trivial bundle, then it seems less obvious how to proceed. The key step is to pull back to a larger variety $B'\to B$, where $B'$ is an open subset of an affine bundle over $B$. This larger variety will have the same Chow ring in low codimension and will also map to $\on{Gr}(e,N)\times \on{Gr}(f,N)$. This is carried out in the proof of \cite[Proposition 19]{EG98} and spelled out in \cite[Section 2.1]{PGL2}.
\end{rmk}

\subsection{Proof of \Cref{dualloci} up to sign}

Given \Cref{lem:equivariantreduction}, one can apply localization in equivariant intersection theory to compute the class of $[D]\in A^{\bullet}_{GL_e\times GL_f}(\on{Hom}(\on{Sym}^m\mb{C}^e,\mb{C}^f))$. We will do this in \Cref{sec:localization}. 

In this section, we present a simple, elementary argument suggested by Anand Patel that gives \Cref{dualloci} up to sign, when combined with \Cref{lem:equivariantreduction}. This argument is independent of \Cref{sec:localization}, strong enough for the applications to slope in \Cref{main}, and we feel it is useful for those unfamiliar with localization.

\begin{thm}
\label{duallociweak}
Let $D\subset \on{Hom}(\on{Sym}^m\mb{C}^e,\mb{C}^f)$ be a divisor, preserved under the natural actions of $GL(\mb{C}^e)$ and $GL(\mb{C}^f)$. Given vector bundles $\mc{E}$ and $\mc{F}$ of ranks $e$ and $f$ respectively over a scheme $B$ together with a map $\phi: \on{Sym}^m\mc{E}\to \mc{F}$, the class of the virtual divisor supported on points of $B$ over which $\phi$ fiberwise restricts to maps in $D$ is a rational multiple of
\begin{align*}
    ec_1(\mc{F})-mfc_1(\mc{E}).
\end{align*}
\end{thm}

\Cref{duallociweak} is weaker than \Cref{dualloci} since it only says the class is a multiple of $ec_1(\mc{F})-mfc_1(\mc{E})$ instead of a positive multiple. Positivity can also be proven directly in this setting but we feel the proof is cleaner when viewed equivariantly.

\begin{proof}
In the context of \Cref{duallociweak}, let $\mc{D}\subset \on{Hom}(\on{Sym}^m\mb{C}^e,\mb{C}^f)$ be the divisor that restricts fiberwise to $D$. The map $\phi$ induces a section $\phi: B\to \on{Hom}(\on{Sym}^m\mc{E},\mc{F})$, and the virtual divisor given in \Cref{duallociweak} is the pullback of $\mc{D}$ to $B$. Thus, it suffices to compute the class of $\mc{D}$. 

By \Cref{lem:equivariantreduction}, the class of $\mc{D}$ can be written as $ac_1(\mc{E})+bc_1(\mc{F})$ for some integers $a$ and $b$. It suffices to compute the ratio of $a$ and $b$. Now, let $\mc{L}$ be a line bundle on $B$. Then,
\begin{align*}
    \on{Hom}(\on{Sym}^m\mc{E},\mc{F})\cong \on{Hom}(\on{Sym}^m(\mc{E}\otimes \mc{L}),\mc{F}\otimes \mc{L}^{\otimes m}),
\end{align*}
and the divisor $\mc{D}$ respects this isomorphism. Therefore, we must have
\begin{align*}
    ac_1(\mc{E})+bc_1(\mc{F}) &= ac_1(\mc{E}\otimes\mc{L})+bc_1(\mc{F}\otimes \mc{L}^{\otimes m}),
\end{align*}
meaning $aec_1(\mc{L})+bfmc_1(\mc{L})=0$. Now, choosing for example, $\mc{L}$ to be $\mathscr{O}(1)$ and $B=\mb{P}^1$ shows $ae+bmf=0$. Thus, there is some constant $c$ such that $a=-cmf$ and $b=ce$, finishing the proof.
\end{proof}

\subsection{Proof of \Cref{dualloci} via localization}
\label{sec:localization}

Given \Cref{lem:equivariantreduction}, \Cref{dualloci} follows from \Cref{quot} below. 

\label{equivariant}
\begin{lem}
\label{principle}
Let $T$ be a torus acting on an affine space $\mb{A}^N$. Then, the equivariant Chow ring $A^{\bullet}_T(\mb{A}^N)\cong \mb{Z}[t_1,\ldots,t_n]$, where $t_1,\ldots,t_n$ $\mb{Z}$-linearly span the character lattice of $T$. 

If $D\subset \mb{A}^N$ is a $T$-invariant divisor, then it is defined by a polynomial $F(x_1,\ldots,x_N)$ whose monomials have the same weight $\chi$ under the action of $T$. The equivariant class $[D]\in A^{\bullet}_T(\mb{A}^N)$ is $\chi$. 
\end{lem}

\begin{proof}
The statement on $A^{\bullet}_T(\mb{A}^N)\cong \mb{Z}[t_1,\ldots,t_n]$ is standard \cite[Section 3.1]{EG98}. The statement on the class of $[D]$ is used in \cite[Theorem 5.1]{Quadric} and can be proven for example by scaling the coordinates of $\mb{A}^N$ to degenerate to the case where $D$ is defined by a monomial. Then, we reduce to the case where $F$ is simply a coordinate function of $\mb{A}^N$. 
\end{proof}

\begin{lem}
\label{quot}
If $D\subset \on{Hom}(\on{Sym}^m\mb{C}^e,\mb{C}^f)$ is a divisor, preserved under the natural actions of $GL(\mb{C}^e)$ and $GL(\mb{C}^f)$, then the equivariant class $$[D]\in A^{1}_{GL(\mb{C}^e)\times GL(\mb{C}^f)}(\on{Hom}(\on{Sym}^m\mb{C}^e,\mb{C}^f))$$ is a positive multiple of
\begin{align*}
e\sum \beta_i-mf\sum \alpha_i,
\end{align*}
where $\{\alpha_i\}$ and $\{\beta_i\}$ are the standard characters of the standard maximal tori of $GL(\mb{C}^e)$ and $GL(\mb{C}^f)$ respectively. 
\end{lem}

The proof of \Cref{quot} follows from \Cref{principle}.
\begin{proof}[Proof of \Cref{quot}]
Let $T_e$ and $T_f$ be the standard maximal tori of $GL(\mb{C}^e)$ and $GL(\mb{C}^f)$ respectively. The restriction map 
$$A^{1}_{GL(\mb{C}^e)\times GL(\mb{C}^f)}(\on{Hom}(\on{Sym}^m\mb{C}^e,\mb{C}^f))\to A^{1}_{T_e\times T_f}(\on{Hom}(\on{Sym}^m\mb{C}^e,\mb{C}^f))$$
is injective \cite[Proposition 6]{EG98}. 

To determine $[D]$ we apply \Cref{principle}. Let $\alpha_1,\ldots,\alpha_e$ be the standard characters of $T_e$ and let $\beta_1,\ldots,\beta_f$ be the standard characters of $T_f$. Viewing $\on{Hom}(\on{Sym}^m\mb{C}^e,\mb{C}^f)$ as the space of $\binom{e+1}{m}\times f$ matrices, $T_e$ and $T_f$ act by the characters $\{\beta_i-\sum_{j\in S}{\alpha_j}\}$ on the entries, where $i$ ranges from 1 to $f$ and $S$ ranges over multisets of $\{1,\ldots,e\}$ with size $m$. Each monomial term of the hypersurface $F$ defining $D$ in $\on{Hom}(\on{Sym}^m\mb{C}^e,\mb{C}^f)$ has a certain weight $\chi$.

Now, we use the fact that $\chi$ has to be invariant under permutation of the characters $\alpha_i$ and the characters $\beta_i$, which means that it must be
\begin{align*}
e\sum \beta_i-mf\sum \alpha_i
\end{align*}
up to a positive power. 
\end{proof}

Now, we can prove \Cref{dualloci}. 
\begin{proof}[Proof of \Cref{dualloci}]
By \Cref{quot}, the equivariant class
$$[D]\in A^{1}_{GL(\mb{C}^e)\times GL(\mb{C}^f)}(\on{Hom}(\on{Sym}^m\mb{C}^e,\mb{C}^f))\cong \mb{Z}[c_1^e,\ldots,c_e^e]\otimes_{\mathbb{Z}}\mathbb{Z}[ c_1^f,\ldots,c_f^f]$$ 
is a positive rational multiple of $ec_1^f-mfc_1^e$. Let $\mc{D}$ be the divisor $\mc{D}\subset \on{Hom}(\on{Sym}^m\mc{E},\mc{F})$ that restricts fiberwise to $D$. \Cref{lem:equivariantreduction} says the class of $\mc{D}$ is a positive rational multiple of $ec_1(\mc{F})-mfc_1(\mc{E})$. Pulling back $[\mc{D}]$ under the section $\phi: B\to \on{Hom}(\on{Sym}^m\mc{E},\mc{F})$
in the statement of \Cref{dualloci} finishes the proof. 
\end{proof}

\section{Application to Slopes of $\ol{\mc{M}}_g$}
\subsection{Setup}
\label{setup}
In addition to \Cref{dualloci}, we will need to pushforward classes from the moduli stack parameterizing a genus $g$ curve together with a $\mf{g}^{r}_d$. The key ingredients were first written in \cite{Khosla} and \cite[Section 2]{Koszul}. The details of the setup will not be used, and the same setup as already been utilized for computations in \cite{Quadric, Khosla, Koszul,C12}. We will follow \cite[Section 5.1]{C12}.

As a first approximation, we want a stack $\wt{\mc{G}^r_d}$ parameterizing curves with a choice of $\mf{g}^{r}_d$ together with a proper map $\wt{\mc{G}^r_d}\to \ol{\mc{M}_g}$. In order to be able to define the universal line bundle and vector bundle corresponding to choice of sections over $\wt{\mc{G}^r_d}$, we will work instead with $\ol{\mc{M}_{g,1}}$. (This is not strictly necessary, also see the second page of \cite[Section 2]{Koszul}.)

Recall for $g\geq 3$ \cite[Theorem 2]{AC87},
\begin{align*}
    A^{\bullet}(\ol{\mc{M}_{g,1}})\otimes \mb{Q}=\mb{Q}\lambda\oplus \bigoplus_{i=0}^{g-1}{\mb{Q}\delta_i}\oplus \mb{Q}\psi,
\end{align*}
where $\delta_0$ is the class of the irreducible nodal curves $\Delta_0\subset \ol{\mc{M}_{g,1}}$, and $\delta_i$ for $i\geq 1$ is the class of the closure of the reducible nodal curves $\Delta_i\subset \ol{\mc{M}_{g,1}}$ where the component containing the marked point is genus $i$. Also, $\lambda$ is the first chern class of the Hodge bundle and $\psi$ is the relative dualizing sheaf of $\ol{\mc{M}_{g,1}}\to \ol{\mc{M}_{g}}$.

We restrict to an open substack $\wt{\mc{M}_{g,1}}\subset \ol{\mc{M}_{g,1}}$ whose compliment is codimension 2, so this step does not affect divisor calculations. Specifically, we first let $\wt{\mc{M}_{g,1}}$ be the complement of the closure of the locus of two smooth curves intersecting transversely at two points.

There is a Deligne-Mumford stack $\mc{G}^r_d\to \wt{\mc{M}_{g,1}}$ parameterizing the choice of a curve $C$, a rank 1 torsion free sheaf $L$, and an $r+1$-dimensional subspace of the global sections of the sheaf. The torsion free sheaf $L$ is restricted to have degree $d$ on the component of $C$ containing the marked point and zero on the unmarked components. Let $\pi: \mc{C}^{r}_d\to \mc{G}^r_d$ be the universal (quasi-stable) curve. Equivalently, $\mc{C}^{r}_d\to \mc{G}^r_d$ is the pullback of the universal curve over $\wt{\mc{M}_{g}}$ under $\mc{G}^r_d\to \wt{\mc{M}_{g,1}}\to \wt{\mc{M}_{g}}$.

On $\mc{C}^{r}_d$, there is a universal sheaf $\mc{L}$ whose restriction to each fiber of $\pi$ is a torsion-free sheaf with degree $d$ on the component with the marked point and degree zero on the other components. Furthermore, $\mc{L}$ is normalized to be trivial along the marked section of $\pi$. In addition, there is a subbundle $\mc{V}\to \pi_{*}\mc{L}$ that restricts to the marked aspect of the (limit) linear series in each fiber. 

We want to apply \Cref{dualloci} in the case where $\mc{E}=\mc{V}$ and $\mc{F}=\pi_{*}\mc{L}^{\otimes m}$. To do, we need $c_1(\pi_{*}\mc{L}^{\otimes m})$ and we need to know $\pi_{*}\mc{L}^{\otimes m}$ is locally free away from a set of codimension 2.

Unfortunately, $\pi_{*}\mc{L}^{\otimes m}$ jumps in rank over $\Delta_i$ for $i>0$. Therefore, we restrict $\mc{G}^r_d\to \wt{\mc{M}_{g,1}}$ to $\mc{G}^{r,\on{irr}}_d\to \mc{M}^{\on{irr}}_{g,1}$, where $\mc{M}^{\on{irr}}_{g,1}\subset \wt{\mc{M}_{g,1}}$ is the complement of $\Delta_i$ for $i>0$ and $\mc{G}^{r,\on{irr}}_d$ is the inverse image of $\mc{M}^{\on{irr}}_{g,1}$ in $\mc{G}^r_d$. 

Then, $A^{\bullet}(\mc{M}^{\on{irr}}_{g,1})\otimes \mb{Q}=\mb{Q}\lambda\oplus \mb{Q}\delta_0\oplus \mb{Q}\psi$, which means we cannot compute the coefficients of $\delta_i$ for $i>0$. Conjecturally this does not matter for computing the slope of $\ol{\mc{M}_g}$ \cite[Conjecture 1.5]{FP05}.
\subsection{Computation}
By an abuse of notation, let us also refer to the restriction $\mc{C}^{r,\on{irr}}_d\to \mc{G}^{r,\on{irr}}_d$ of $\mc{C}^{r}_d\to \mc{G}^r_d$ as $\pi$ and let $\omega$ be the dualizing sheaf of $\pi$. Then, following \cite{Khosla,Koszul}, we define
\begin{align*}
    \alpha = \pi_{*}(c_1(\mc{L})^2) \qquad \beta = \pi_{*}(c_1(\mc{L})\cap c_1(\omega)) \qquad \gamma = c_1(\mc{V}),
\end{align*}
where $\mc{L}$ and $\mc{V}$ are restricted to $\mc{C}^{r,\on{irr}}_d$ and $\mc{G}^{r,\on{irr}}_d$ respectively. Let $\eta$ be the map $\mc{G}^{r,\on{irr}}_d\to \mc{M}^{\on{irr}}_{g,1}$. 

In order to have $\rho = g-(r+1)(g-d+r)=0$, $g$ needs to be $s(r+1)$ for some $s>1$. Solving for $d$, we have $d=r(s+1)$. Finally, for $(C,L)\in \mc{G}^{r,\on{irr}}_d$ general, we need $\dim(\on{Sym}^{m}H^0(L))\geq \dim(H^0(L^{\otimes m}))$. If $C$ is general, then the Geiseker-Petri theorem implies $h^1(L^{\otimes 2})=0$, so we must require
\begin{align}
\label{NH}
    \binom{r+m}{m}\geq md-g+1.
\end{align}

The following lemma is already contained in \cite[Section 3A]{Khosla}, but we include it for completeness and to correct a typo in the proof. 
\begin{lem}
\label{GRR}
We have $\pi_{*}\mc{L}^{\otimes m}$ is a vector bundle away from a set of codimension at least 2 and $c_1(\pi_{*}\mc{L}^{\otimes m})=\frac{m^2}{2}\alpha - \frac{m}{2}\beta + \eta^{*}(\lambda)$.
\end{lem}

\begin{proof}
We first claim that for $(C,L)\in \mc{G}^{r,\on{irr}}_d$, then $h^1(L^{\otimes m})=0$ for degree reasons away from a set of codimension at least 2. This implies $R^1\pi_{*}\mc{L}^{\otimes m}=0$ and $\pi_{*}\mc{L}^{\otimes m}$ is a vector bundle away from a set of codimension at least 2 by Grauert's theorem. First, suppose $C$ is smooth. If $m=2$, then $2d-2g+2=2(r-s+1)$. This is greater than zero as $s\leq \frac{r}{2}$ (which is equivalent to \eqref{NH} when $m=2$). If $m\geq 3$, we note 
$$md-2g-2\geq 3rs+3r-2rs-2s+2=rs+3r-2s+2=(r-2)(s+3)+8\geq 0.$$ 
Now, if $C$ is a general irreducible nodal curve, then \cite[proof of Proposition 2.3 (2)]{Koszul}, together with our assumption that the Brill-Noether number is zero, says that $L$ is locally free, and we can repeat the same argument above to see $h^1(L^{\otimes m})=0$. 

To apply Grothendieck Riemann-Roch, we need the Todd class of $\pi$. This is pulled back from the Todd class of $\ol{\mc{M}_{g,1}}\to \ol{\mc{M}_{g}}$, which is computed in \cite[page 158]{HM98}. Applying Grothendieck Riemann-Roch yields
\begin{align*}
    c_1(\pi_{*}\mc{L}^{\otimes m})&= \pi_{*}\left[(1+mc_1(\mc{L})+\frac{m^2}{2}c_1(\mc{L})^2)(1-\frac{1}{2}c_1(\omega)+\frac{c_1(\omega)^2+[Z]}{12}\right]_{2}\\
    &= \frac{m^2}{2}\alpha - \frac{m}{2}\beta+\eta^{*}\frac{\pi_{*}c_1(\omega)^2+\delta}{12}
\end{align*}
where $Z\subset \mc{C}^{r,\on{irr}}_d$ is the singular locus of $\pi: \mc{C}^{r,\on{irr}}_d\to \mc{G}^{r,\on{irr}}_d$. At this point, we use the fact that the universal curve $\pi: \mc{C}^{r,\on{irr}}_d\to \mc{G}^{r,\on{irr}}_d$ is pulled back from $\mc{G}^r_d\to \wt{\mc{M}_{g,1}}\to \wt{\mc{M}_{g}}$. This means $\pi_{*}c_1(\omega)^2$ is the pullback of the $\kappa$ divisor class on $\wt{\mc{M}_{g}}$ under $\mc{G}^r_d\to \wt{\mc{M}_{g,1}}\to \wt{\mc{M}_{g}}$. Using the relation $\frac{\kappa+\delta}{12}=\lambda$ \cite[page 158]{HM98}, we have $\eta^{*}\frac{\pi_{*}c_1(\omega)^2+\delta}{12}=\eta^{*}\lambda$, resulting in the claimed formula in \Cref{GRR}. 
\end{proof}

\begin{thm}
[{\cite[Theorem 2.11]{Khosla}}]
\label{pushforward}
Choose $g,r,d\geq 1$ integers such that $\rho = g-(r+1)(g-d+r)=0$.
Then, pushing forward under $\eta: \mc{G}^{r,\on{irr}}_d\to \mc{M}^{\on{irr}}_{g,1}$, we have
\begin{align*}
    \frac{6(g-1)(g-2)}{dN} \eta_{*}\alpha =&6 (gd - 2 g^2 + 8 d - 8 g + 4) \lambda + (2g^2 - gd + 3g - 4d - 
    2) \delta_0  \\&- 6 d (g - 2) \psi\\
    \frac{2 (g - 1)}{Nd}\eta_{*}\beta =& 12\lambda - \delta_0  - 2 (g - 1) \psi\\
    \frac{2(g-1)(g-2)}{N}\eta_{*}\gamma =& ((-(g + 3) \xi + 5 r (r + 2)) \lambda - d (r + 1) (g - 2) \psi + \\*&  \frac{1}{6} ((g + 1) \xi - 3 r (r + 2)) \delta_0 ),
\end{align*}
\end{thm}
where 
\begin{align*}
    N&=\frac{g!\prod_{i=1}^r i!}{\prod_{i=0}^{r}(g-d+r+i)}(=\deg(\eta))\\
    \xi &= 3(g-1)+\frac{(r-1)(g+r+1)(3g-2d+r-3)}{g-d+2r+1}.
\end{align*}

\begin{proof}[Proof of \Cref{main}]
Following the notation of \Cref{setup}, apply \Cref{dualloci} in the case where $\mc{E}=\mc{V}$ and $\mc{F}=\pi_{*}\mc{L}^{\otimes k}$ and $\pi: \mc{C}^{r,\on{irr}}_d\to \mc{G}^{r,\on{irr}}_d$. This yields a positive multiple of
\begin{align}
\label{class}
    (r + 1)(\frac{m^2}{2}\alpha - \frac{m}{2}\beta + \sigma^{*}\lambda) - m(md - g + 1) \gamma
\end{align}
on $\mc{G}^{r,\on{irr}}_d$. Since we only care about the slope, we can scale by a constant factor and work with \eqref{class}. We push forward \eqref{class} via $\eta: \mc{G}^{r,\on{irr}}_d\to \mc{M}^{\on{irr}}_{g,1}$ using \Cref{pushforward} to get a class $a\lambda+b_0\delta_0+c\psi$. This yields $c=0$ (as expected) and rather complicated formulas for $a$ and $b_0$. Checking these formulas using Mathematica yields the three statements of \Cref{main}. For more details, the interested reader can refer to \Cref{grind}.
\end{proof}

\section{Acknowledgements}
The author would like to thank Joe Harris for introducing the author to moduli spaces of curves and for helpful discussions. The author would also like to thank Aleksei Kulikov for the reference \cite{F08} on the density of the candidate counterexamples to the Slope Conjecture, Anand Patel for the alternative proof of \Cref{dualloci} presented, and Gavril Farkas for helpful comments.
\appendix
\section{Mathematica computation}
\label{grind}
\begin{proof}[Proof of \Cref{main} continued]

Continuing the proof of \Cref{main}, we find
\begin{polynomial}
a=\frac{N}{2 (r+s+1) (r s+s-2) (r s+s-1)}(m^2 r^5 s^4-m^2 r^5 s^2+3 m^2 r^4 s^4+5 m^2 r^4 s^3+m^2 r^4 s^2-m^2 r^4 s+m^2 r^3 s^4+12 m^2 r^3 s^3+13 m^2 r^3 s^2-2 m^2 r^3 s-4 m^2 r^3-3 m^2 r^2 s^4-5 m^2 r^2 s^3+m^2 r^2 s^2+3 m^2 r^2 s-2 m^2 r s^4-12 m^2 r s^3-14 m^2 r s^2+4 m^2 r-m r^5 s^4+m r^5 s^2-5 m r^4 s^4-5 m r^4 s^3-m r^4 s^2+m r^4 s-9 m r^3 s^4-26 m r^3 s^3-13 m r^3 s^2+22 m r^3 s+4 m r^3-7 m r^2 s^4-37 m r^2 s^3-5 m r^2 s^2+57 m r^2 s-2 m r s^4-16 m r s^3+6 m r s^2+40 m r s-4 m r+2 r^4 s^2+2 r^3 s^3+8 r^3 s^2-6 r^3 s+6 r^2 s^3+6 r^2 s^2-18 r^2 s+4 r^2+6 r s^3-4 r s^2-14 r s+8 r+2 s^3-4 s^2-2 s+4)\\
b_0=-\frac{N}{12 (r+s+1) (r s+s-2) (r s+s-1)}m r (r+1) (s+1) (m r^3 s^3-m r^3 s^2+2 m r^2 s^3+m r^2 s^2-m r s^3+5 m r s^2+m r s-2 m r-2 m s^3-5 m s^2-m s+2 m-r^3 s^3+r^3 s^2-4 r^2 s^3+r^2 s^2-5 r s^3-7 r s^2+7 r s+2 r-2 s^3-7 s^2+17 s-2)
\end{polynomial}
This yields $-\frac{a}{b_0}$ as a complicated rational function $F(m,r,s)$ for the slope $s_0(D_{m,r,s})$. We now prove each case individually. Recall in each case $g=(r+1)s$ and $d=r(s+1)$. 
\begin{proof}[Proof of Part 1 of \Cref{main}]
Consider $F(m,r,s)-(6+\frac{12}{g+1})$. This again is a complicated rational function $G(m,r,s)$ in $m,r,s$. To see $G(m,r,s)\geq 0$ if $m\geq 3$, $r\geq 3$, $s\geq 1$ subject to the constraint $\binom{r+m}{m}-(dm-g+1)\geq 0$, we first note that \begin{polynomial}
G(m+4,r+4,s+1)=(6 (6+r+s) (3+r+5 s+r s) (4+r+5 s+r s) (156+120 m+24 m^2+110 r+78 m r+14 m^2 r+18 r^2+12 m r^2+2 m^2 r^2+2 s+36 m s+12 m^2 s+24 r s+29 m r s+7 m^2 r s+6 r^2 s+5 m r^2 s+m^2 r^2 s))/((4+m) (4+r) (5+r) (2+s) (6+r+5 s+r s) (162+54 m+81 r+27 m r+9 r^2+3 m r^2+513 s+207 m s+330 r s+120 m r s+58 r^2 s+20 m r^2 s+3 r^3 s+m r^3 s+543 s^2+237 m s^2+410 r s^2+154 m r s^2+89 r^2 s^2+31 m r^2 s^2+6 r^3 s^2+2 m r^3 s^2+210 s^3+90 m s^3+167 r s^3+63 m r s^3+40 r^2 s^3+14 m r^2 s^3+3 r^3 s^3+m r^3 s^3))
\end{polynomial}
\noindent is clearly positive. To deal with the edge cases when $m=3$ or $r=3$, we first find
\begin{polynomial}
G(3,r+3,s+1)=(2 (5+r+s) (2+r+4 s+r s) (3+r+4 s+r s) (11+15 r+4 r^2-11 s-r s+r^2 s))/((3+r) (4+r) (2+s) (5+r+4 s+r s) (30+21 r+3 r^2+66 s+70 r s+16 r^2 s+r^3 s+52 s^2+76 r s^2+23 r^2 s^2+2 r^3 s^2+20 s^3+29 r s^3+10 r^2 s^3+r^3 s^3)).
\end{polynomial}
The only factor of $G(3,r+3,s+1)$ that can be negative is $(11 + 15 r + 4 r^2 - 11 s - r s + r^2 s)$. Now, we use the constraint $\binom{r+m}{m}-(dm-g+1)\geq 0$. Substituting $m\to 3$ yields $\frac{r^3}{6}+r^2-2 r s-\frac{7 r}{6}+s\geq 0$, so $s\leq \frac{r^3+6 r^2-7 r}{6 (2 r-1)}$. Plugging in $s=\frac{(r+3)^3+6 (r+3)^2-7 (r+3)}{6 (2 (r+3)-1)}-1$ into $(11 + 15 r + 4 r^2 - 11 s - r s + r^2 s)$ yields $\frac{r (r+1) (r+4) \left(r^2+9 r+17\right)}{6 (2 r+5)}$, which is nonnegative. Furthermore, this is zero only when $r=0$. Therefore, we have $G(3,r,s)\geq 0$ for $r\geq 3, s\geq 1$ and equality can hold only if $r=3$. In this case, $s\leq \frac{r^3+6 r^2-7 r}{6 (2 r-1)}=2$. Plugging in $s=1,2$ yields $G(3,3,1)>0$ and $G(3,3,2)=0$. 

Now, we are left with the case $r=3$, $m\geq 4$ and $s\geq 1$. Note
\begin{polynomial}
    G(m+5,3,s+1)= ((5+s) (1+2 s) (3+4 s) (65+39 m+6 m^2+s+12 m s+3 m^2 s))/((5+m) (2+s) (5+4 s) (60+15 m+172 s+53 m s+164 s^2+56 m s^2+60 s^3+20 m s^3))
\end{polynomial}
\noindent is clearly positive, so we are left with the case $r=3$, $m= 4$ and $s\geq 1$. Since $$\binom{3+4}{4}-(4d-g+1)\geq 0\Leftrightarrow 22-8s\geq 0,$$
so our remaining candidates are $(m,r,s)=(4,3,1)$ or $(4,3,2)$. We evaluate
\begin{align*}
    G(4,3,s+1)&= -\frac{2 (s-5) (s+4) (2 s-1) (4 s-1)}{(s+1) (4 s+1) \left(40 s^3-12 s^2+23 s-6\right)}
\end{align*}
and note that it is positive for $s=1,2$. Tracing through the cases, we find $G(m,r,s)\geq 0$ for $m\geq 3,r\geq 3,s\geq 1$ subject to the constraint $\binom{r+m}{m}-(dm-g+1)\geq 0$, and equality holds when $(m,r,s)=(3,3,2)$. 
\end{proof}
\begin{proof}[Proof of Part 2 of \Cref{main}]
Define $G(m,r,s)=F(m,r,s)-(6+\frac{8}{g+1})$. We want to see $G(m,r,s)> 0$ if $m\geq 2$, $r\geq 3$, $s\geq 1$ subject to the constraint $\binom{r+m}{m}-(dm-g+1)\geq 0$. First note
\begin{polynomial}
G(m+2,r+5,s+1)=(2 (30240+51240 m+26880 m^2+27168 r+48018 m r+25176 m^2 r+9774 r^2+17994 m r^2+9390 m^2 r^2+1770 r^3+3378 m r^3+1746 m^2 r^3+162 r^4+318 m r^4+162 m^2 r^4+6 r^5+12 m r^5+6 m^2 r^5+76896 s+179100 m s+102960 m^2 s+79264 r s+171950 m r s+94152 m^2 r s+31166 r^2 s+64661 m r^2 s+34067 m^2 r^2 s+5928 r^3 s+11947 m r^3 s+6101 m^2 r^3 s+550 r^4 s+1087 m r^4 s+541 m^2 r^4 s+20 r^5 s+39 m r^5 s+19 m^2 r^5 s+61560 s^2+213960 m s^2+132360 m^2 s^2+79312 r s^2+211992 m r s^2+119522 m^2 r s^2+35270 r^2 s^2+81050 m r^2 s^2+42560 m^2 r^2 s^2+7224 r^3 s^2+15042 m r^3 s^2+7470 m^2 r^3 s^2+700 r^4 s^2+1360 m r^4 s^2+646 m^2 r^4 s^2+26 r^5 s^2+48 m r^5 s^2+22 m^2 r^5 s^2+21048 s^3+109500 m s^3+68280 m^2 s^3+36976 r s^3+112180 m r s^3+61426 m^2 r s^3+18634 r^2 s^3+43891 m r^2 s^3+21769 m^2 r^2 s^3+4096 r^3 s^3+8282 m r^3 s^3+3798 m^2 r^3 s^3+416 r^4 s^3+758 m r^4 s^3+326 m^2 r^4 s^3+16 r^5 s^3+27 m r^5 s^3+11 m^2 r^5 s^3+7056 s^4+24120 m s^4+12240 m^2 s^4+10200 r s^4+24564 m r s^4+11028 m^2 r s^4+4828 r^2 s^4+9598 m r^2 s^4+3916 m^2 r^2 s^4+1034 r^3 s^4+1815 m r^3 s^4+685 m^2 r^3 s^4+104 r^4 s^4+167 m r^4 s^4+59 m^2 r^4 s^4+4 r^5 s^4+6 m r^5 s^4+2 m^2 r^5 s^4))/((2+m) (5+r) (6+r) (2+s) (7+r+6 s+r s) (84+84 m+33 r+33 m r+3 r^2+3 m r^2+208 s+348 m s+129 r s+163 m r s+21 r^2 s+23 m r^2 s+r^3 s+m r^3 s+200 s^2+424 m s^2+162 r s^2+222 m r s^2+33 r^2 s^2+37 m r^2 s^2+2 r^3 s^2+2 m r^3 s^2+84 s^3+168 m s^3+68 r s^3+94 m r s^3+15 r^2 s^3+17 m r^2 s^3+r^3 s^3+m r^3 s^3)),
\end{polynomial}
\noindent which is clearly positive. This leaves the cases when $r=3$ and $r=4$. To deal with the case $r=4$, we evaluate
\begin{polynomial}
G(m+2,4,s+2)=(7560+31584 m+20832 m^2+7561 s+55078 m s+37970 m^2 s+2083 s^2+35878 m s^2+25162 m^2 s^2+335 s^3+10610 m s^3+7190 m^2 s^3+125 s^4+1250 m s^4+750 m^2 s^4)/(5 (2+m) (3+s) (11+5 s) (84+196 m+109 s+317 m s+53 s^2+169 m s^2+10 s^3+30 m s^3))\\
G(m+2,4,1)=\frac{2 \left(11 m^2+21 m+13\right)}{15 (m+1) (m+2)}.
\end{polynomial}
To deal with the case $r=3$, we evaluate 
\begin{polynomial}
G(m+3,3,s+1)=(885+825 m+210 m^2+2362 s+2895 m s+847 m^2 s+2007 s^2+3410 m s^2+1119 m^2 s^2+642 s^3+1640 m s^3+582 m^2 s^3+152 s^4+320 m s^4+104 m^2 s^4)/((3+m) (2+s) (5+4 s) (30+15 m+66 s+53 m s+52 s^2+56 m s^2+20 s^3+20 m s^3)),
\end{polynomial}
\noindent which reduces us to the case $r=3$, $m=2$. Now, we use the bound $\binom{3+2}{2}-(2d-g+1)\geq 0\Leftrightarrow 3-2s\geq 0$. Plugging in $G(2,3,1)>0$ finishes this case. 
\end{proof}
\begin{proof}[Proof of Part 3 of \Cref{main}]
Define $G(m,r,s)=F(m,r,s)-(6+\frac{12}{g+1})$. We want to see when $G(m,r,s)< 0$ if $m=2$, $r\geq 3$, $s\geq 1$ subject to the constraint $\binom{r+m}{m}-(dm-g+1)\geq 0$. First, note 
\begin{align*}
    \binom{r+m}{m}-(dm-g+1)\geq 0\Leftrightarrow s\leq \frac{r}{2},
\end{align*}
which is one of the constraints claimed in Part 3 of \Cref{main}. Next, we evaluate 
\begin{polynomial}
G(2,r,s)=(6 (1 + r + s) (1 + r^2 + s - 3 r s) (-2 + s + r s) (-1 + s + 
   r s))/(r (1 + r) (1 + s) (1 + s + r s) (2 - 2 r + 15 s + 9 r s - 
   17 s^2 + 3 r s^2 + 3 r^2 s^2 - r^3 s^2 - 6 s^3 - 7 r s^3 + r^3 s^3)
  )\\
G(2,3,s+1)=\frac{(s+5) (2 s+1) (4 s-1) (4 s+3)}{(s+2) (4 s+5) \left(4 s^2-13 s-15\right)}\\
G(2,r+4,s+1)=(6 (6+r+s) (6+5 r+r^2-11 s-3 r s) (3+r+5 s+r s) (4+r+5 s+r s))/((4+r) (5+r) (2+s) (6+r+5 s+r s) (54+27 r+3 r^2+99 s+90 r s+18 r^2 s+r^3 s+69 s^2+102 r s^2+27 r^2 s^2+2 r^3 s^2+30 s^3+41 r s^3+12 r^2 s^3+r^3 s^3)).
\end{polynomial}
\noindent Therefore, if $r\geq 4$, then $G(2,r,s)<0$ if and only if 
\begin{align*}
    1 + r^2 + s - 3 r s<0\Leftrightarrow s>\frac{r^2+1}{3 r-1}.
\end{align*}
If $r=3$, then $s\leq \frac{3}{2}$, so $s=1$ and $G(2,3,1)>0$.
\end{proof}
\end{proof}

\section{Classification of the invariant divisors}
\label{classdivisors}
Given vector spaces $V$ and $W$, it is natural to ask for a classification of divisors $D\subset \on{Hom}(\on{Sym}^m(V),W)$ invariant under the action of $GL(V)\times GL(W)$. First, let us argue that these invariant divisors actually exist. 

\begin{prop}
If $V$ and $W$ are vector spaces with $\dim(V)\geq 4$ and $\dim(\on{Sym}^m V)\geq \dim(W)$, then there exists a divisor $D\subset \on{Hom}(\on{Sym}^m(V), W)$ that is invariant under the action of $GL(V)\times GL(W)$.
\end{prop}

\begin{proof}
If $\dim(\on{Sym}^m(V))= \dim(W)$, then $D$ can be chosen to be the linear maps not of full rank. In fact, this is the unique choice for $D$ in this case, as the $GL(W)$ orbit of a nonsingular matrix is all nonsingular matrices in the space of square matrices.

If $\dim(\on{Sym}^m(V))= \dim(W)+1$, then we can choose $D$ to consist of linear maps whose kernel contains a nonzero homogenous form defining a singular hypersurface in $\mathbb{P}(V^{*})$. 

If $\dim(\on{Sym}^m(V))\geq \dim(W)+2$, then a computation shows a general $GL(V)\times GL(W)$ orbit is codimension at least 1 in $\on{Hom}(\on{Sym}^m(V),W)$ for dimension reasons, so we can let $D$ be the closure of a union of a family of $GL(V)\times GL(W)$ orbits. 

More explicitly, let $\Lambda_1$ be a general 1-dimensional vector subspace of $\on{Hom}(\on{Sym}^m(V),W)$. The codimension of $(GL(V)\times GL(W))\cdot \Lambda_1$ is at least
\begin{align*}
\dim(\on{Hom}(\on{Sym}^{m}V,W))-\dim(GL(V))-\dim(GL(W))+1&=\\
\left(\binom{\dim(V)-1+m}{m}-\dim(W)\right)\dim(W)-\dim(V)^2+1&\geq\\
\left(\binom{\dim(V)+1}{2}-2\right)2-\dim(V)^2+1&=\dim(V)-3\geq 1.
\end{align*}
If the codimension of $(GL(V)\times GL(W))\cdot \Lambda_1$ is precisely 1, then we can let $D$ be its closure and we have produced a $GL(V)\times GL(W)$-invariant divisor inside $\on{Hom}(\on{Sym}^m(V),W)$. Otherwise, we let $\Lambda_2$ be the span of $\Lambda_1$ together with a general point of $\on{Hom}(\on{Sym}^m(V),W)$. Then, $(GL(V)\times GL(W))\cdot \Lambda_2$ has dimension precisely one greater than the dimension of $(GL(V)\times GL(W))\cdot \Lambda_1$. Repeating this process, we eventually obtain a general linear space $\Lambda_i\subset \on{Hom}(\on{Sym}^m(V),W)$ such that the closure of the union of orbits $(GL(V)\times GL(W))\cdot \Lambda_i$ is a $GL(V)\times GL(W)$-invariant divisor inside $\on{Hom}(\on{Sym}^m(V),W)$.
\end{proof}


Next, if one is interested in showing the divisors $D_{m,r,s}$ defined in \Cref{Definitions} are not all virtual, then it would be good to understand the intersection $Z\subset \on{Hom}(\on{Sym}^m(V),W)$ of all $GL(V)\times GL(W)$ invariant divisors. \Cref{prop:GIT} follows essentially by definition.

\begin{prop}
\label{prop:GIT}
The intersection $Z$ of all $GL(V)\times GL(W)$ invariant divisors $D\subset \on{Hom}(\on{Sym}^m(V), W)$ is the locus of GIT unstable points in $\on{Hom}(\on{Sym}^m(V), W)$ under the action of the subgroup $SL(V)\times GL(W)$, where the trivial bundle is linearized by the character sending $(A,B)\in SL(V)\times GL(W)$ to $\det(B)$. 
\end{prop}
 
\begin{proof}
A divisor on the affine space $\on{Hom}(\on{Sym}^m(V), W)$ is given by the vanishing locus of a polynomial $F$ in $\dim(\on{Sym}^m(V))\dim(W)$ variables, where the action of $GL(V)\times GL(W)$ on $F$ is by a character. Having the group $GL(V)\times GL(W)$ acting on $F$ by a character is equivalent to the subgroup $SL(V)\times GL(W)$ acting by a character. 

Let $L$ be the trivial bundle on $\on{Hom}(\on{Sym}^m(V), W)$, where $SL(V)\times GL(W)$ acts on $L$ by the character of $SL(V)\times GL(W)$ sending $(A,B)\in SL(V)\times GL(W)$ to $\det(B)$. Then, all divisors on $\on{Hom}(\on{Sym}^m(V), W)$ are given by the $SL(V)\times GL(W)$-invariant sections $H^0(L^{\otimes m})$ as $m$ ranges over all nonnegative integers. 

The common vanishing locus of all these sections is by definition the locus of GIT unstable points in $\on{Hom}(\on{Sym}^m(V), W)$ under the action of the subgroup $SL(V)\times GL(W)$.
\end{proof}
The special linear group is its own commutator subgroup (except in the case $SL_2(\mathbb{F}_2)\cong S_3$ and $SL_2(\mathbb{F}_3)$) \cite[Theorems 8.2 and 9.3]{L02} and we are working over an algebraically closed field of characteristic zero. Thus, a character on $SL(V)\times GL(W)$ must send $(A,B)\in SL(V)\times GL(W)$ to an integral power of $\det(B)$, so all nontrivial choices of linearization of the trivial bundle on $\on{Hom}(\on{Sym}^m(V), W)$ in the context of GIT in \Cref{prop:GIT} are equivalent. 

To understand the GIT unstable points in \Cref{prop:GIT}, one can first quotient by $GL(W)$ to get the Grassmannian $\on{Gr}(\dim(W),\on{Sym}^m(V))$ of quotients and look at the GIT unstable points under the action of $SL(V)$. The semistable locus of $\on{Gr}(\dim(W),\on{Sym}^m(V))$ under the action of $SL(V)$ has appeared in the study of associated forms, for example in \cite{AI18,F17,FI19}.

\bibliographystyle{alpha}
\bibliography{references.bib}
\end{document}